\documentclass[final,1p,times,authoryear]{elsarticle}
\usepackage[T1]{fontenc}
\usepackage{amsmath}
\usepackage{amsthm}
\usepackage{amssymb}
\usepackage{amsfonts}
\usepackage{listings}
\usepackage{natbib}

\makeatletter

\numberwithin{equation}{section}
\numberwithin{figure}{section}
\theoremstyle{plain}
\newtheorem{thm}{\protect\theoremname}
\theoremstyle{remark}

\theoremstyle{definition}
\newtheorem{defn}[thm]{\protect\definitionname}
\theoremstyle{plain}
\newtheorem{lem}[thm]{\protect\lemmaname}
\theoremstyle{definition}
\newtheorem{example}[thm]{\protect\examplename}
\theoremstyle{remark}
\newtheorem*{rem*}{\protect\remarkname}

\newcommand{\B}{\mathcal{B}}

\newcommand{\U}{\mathcal{U}}
\newcommand{\V}{\mathcal{V}}
\newcommand{\Hc}{\mathcal{H}}
\newcommand{\C}{\mathcal{C}}
\newcommand{\F}{\mathcal{F}}
\newcommand{\G}{\mathcal{G}}
\makeatother

\providecommand{\definitionname}{Definition}
\providecommand{\examplename}{Example}
\providecommand{\lemmaname}{Lemma}
\providecommand{\remarkname}{Remark}
\providecommand{\theoremname}{Theorem}

\begin{document}
\title{Analytic Continuation for Multiple Zeta Values using Symbolic Representations}

\author[DAL]{Lin Jiu}
\ead{lin.jiu@dal.ca}

\author[Maryland]{Tanay Wakhare}
\ead{twakhare@gmail.com}

\author[Tulane,SudOrsay]{Christophe Vignat}
\corref{mycorrespondingauthor}
\cortext[mycorrespondingauthor]{Corresponding author}
\ead{cvignat@tulane.edu}

\address[DAL]{Department of Mathematics and Statistics, 
Dalhousie University, 
6316 Coburg Road, Halifax, NS, Canada B3H 4R2}
\address[Maryland]{University of Maryland, College Park, MD 20742, USA}
\address[Tulane]{Department of Mathematics, Tulane University, New Orleans, LA 70118, USA}
\address[SudOrsay]{LSS/Supelec, Universit\'{e} Paris Sud 11, Orsay, 91192, France}


\begin{abstract}
We introduce a symbolic representation of $r$-fold harmonic sums
at negative indices. This representation
allows us to recover and extend some recent results by Duchamp et
al., such as recurrence relations and generating functions for these
sums. This approach is also applied to the study of the family of extended
Bernoulli polynomials, which appear in the computation of harmonic
sums at negative indices. It also allows us to reinterpret the Raabe analytic continuation
of the multiple zeta function as both a constant term extension of Faulhaber's formula,
and as the result of a natural renormalization procedure for Faulhaber's formula.
\end{abstract}

\maketitle

\section{Introduction}

Faulhaber's classical formula
\[
\sum_{k=1}^{N}k^{n}=\frac{B_{n+1}(N+1)-B_{n+1}}{n+1},\thinspace\thinspace n\ge0,\thinspace\thinspace N\ge1,
\]
exemplifies the importance of Bernoulli polynomials $B_{n}(x)$
in various summation problems. In a recent article \citep{Duchamp},
Duchamp et al.~proposed an extension of this formula to different
types of sums such as the multiple nested sums
\begin{equation}\label{intro}
H_{-n_{1},\dots,-n_{r}}(N) := \sum_{N>i_{1}>\cdots>i_{r}>0}i_{1}^{n_{1}}\cdots i_{r}^{n_{r}},
\end{equation}
and their polynomial version, the truncated polylogarithmic function at negative indices
\begin{equation}\label{poly}
\mathrm{Li}_{-n_{1},\ldots,-n_{r}}(z;N):=\sum_{N>i_{1}>\cdots>i_{r}>0}i_{1}^{n_{1}}\cdots i_{r}^{n_{r}}z^{i_{1}}.
\end{equation}
Their results highlight interesting combinatorial aspects of these
nested sums, together with a natural generalization of the usual Bernoulli
polynomials, called extended Bernoulli polynomials. 

Obviously, both sums \eqref{intro} and \eqref{poly} are related to the \textit{multiple zeta values (MZVs)}
\[
\zeta({s_1,s_2,\ldots,s_r}) = \sum_{n_1>n_2>\cdots>n_r>0} \frac{1}{n_1^{s_1}n_2^{s_2}\cdots n_r^{s_r}},
\]
the study of which is also the ultimate aim for this work. In a previous work \citep{Moll}, we exhibited a natural symbolic approach
to the analytic continuation of multiple zeta values at negative integers. Since $\zeta(s_1,s_2,\ldots,s_r)$ involves multiple variables, \textit{Hartog's phenomenon} indicates that there may be multiple analytic coninuations, in constrast with the one dimensional case. However, we showed that a simple symbolic computation rule
allows us to express in a simple way the analytic continuation of the multiple zeta values
as obtained by Sadaoui \citep[Thm.~1]{Sadaoui} using Raabe's identity. Furthermore, the symbolic approach allowed us to discover that, surprisingly, Raabe's
analytic continuation method generates the same values for the multiple
zeta function at negative integers as those obtained from the Euler-Maclaurin
summation formula \citep[eq.~6]{EulerMaclaurin}. Though these may or may not be identical analytic continuations, the fact that they agree at non-positive integer points is surprising.

In this paper, we demonstrate two more appearances of the same values for the analytic continuation of the multiple zeta function; one is as the constant term of an extension of Faulhaber's formula to nested sums of the form \eqref{intro}, and the other is as the result of a natural renormalization procedure applied to \eqref{intro} in the $N\to \infty$ limit. Although Hartog's phenomenon suggests multiple possible analytic continuations, we wish to highlight the natural appearance of the Raabe--type analytic continuation (again!) as a motivation for its further study.

Encouraged by the simplification that symbolic computation may bring
to the manipulation of complicated sums, in this article we first revisit 
the approach by Duchamp et al.~in order to gain insight on the
significance of the generalized Faulhaber formula and the extended
Bernoulli polynomials that naturally emerge from it. Minh \citep{Minh} has 
obtained results which are similar to ours, though non-symbolic. One of the major
results of the present study is the simple product representation of the multiple nested sums
\[
H_{-n_{1},\dots,-n_{r}}(N)=\prod_{k=1}^{r}\Hc_{1,\ldots,k}^{n_{k}}
\]
from Theorem~\ref{thm:Hgeneralcase}, where the symbols $\Hc_{1,\ldots,k}^{n_{k}}$
are defined below. This simple product representation allows us in
turn to prove several consequences, such as that in Theorem~\ref{thm:The-generating-function} the recurrence
\[
\F_{r}(w_{1},\dots,w_{r};N)=\frac{\F_{r-1}(w_{1},\dots,w_{r-1}+w_{r};N)-\F_{r-1}(w_{1},\dots,w_{r-1};N)}{e^{w_{r}}-1}
\]
satisfied by 
the generating function for the harmonic sums
\[
\F_{r}(w_{1},\dots,w_{r};N):=\sum_{n_{1},\dots,n_{r}=0}^{\infty}\frac{w_{1}^{n_{1}}\cdots w_{r}^{n_{r}}}{n_{1}!\cdots n_{r}!}H_{-n_{1},\dots,-n_{r}}(N)
\]
and
other relations and recurrences for nested harmonic sums. Moreover, these symbolic representations allow us to handle complicated multiple sums related to the analytic continuation of the multiple zeta function.

\section{\label{sec:Symbols}Symbols}

\subsection{The $\B$ and $\U$ symbols}

In what follows, we will frequently use the Bernoulli symbol $\B$
with the evaluation rule
\[
\B^{n}=B_{n},
\]
where $B_{n}$ is the $n$-th Bernoulli number, defined by the generating
function
\[
\sum_{n=0}^{\infty}\frac{B_{n}}{n!}z^{n}=\frac{z}{e^{z}-1}.
\]
Given this definition, we have
\begin{equation}
e^{z\B}=\frac{z}{e^{z}-1},\label{eq:BernoulliGF}
\end{equation}
since
\[
e^{z\B}=\sum_{n=0}^{\infty}\frac{\B^{n}}{n!}z^{n}=\sum_{n=0}^{\infty}\frac{B_{n}}{n!}z^{n}.
\]
Moreover, the Bernoulli polynomials $B_{n}(x)$ with generating
function
\[
\sum_{n=0}^{\infty}\frac{B_{n}(x)}{n!}z^{n}=\frac{z}{e^{z}-1}e^{zx}
\]
then have the simple symbolic expression
\begin{equation}
B_{n}(x)=(\B+x)^{n},\label{eq:BernoulliPolynomial}
\end{equation}
since 
\[
B_{n}(x)=\sum_{k=0}^{n}{n \choose k}B_{n-k}x^{k}=\sum_{k=0}^{n}{n \choose k}\B^{n-k}x^{k}=(\B+x)^{n}.
\]
Additionally, two symbols $\B_{1}$ and $\B_{2}$
are called \emph{independent }if they satisfy, $\forall n,m\in\mathbb{N}$,
\[
\B_{1}^{n}\B_{2}^{m}=B_{n}B_{m}.
\]
Another useful symbol is the uniform symbol $\U$, with the
evaluation rule
\[
\U^{n}=\frac{1}{n+1}.
\]
This is equivalent to the operational action $$f(x+\U) = \int_0^1 f(x+u)\mathrm{d}u.$$
We deduce, for example, 
\[
e^{z\U}=\sum_{n=0}^\infty\frac{z^n\U^n}{n!}=\sum_{n=0}^\infty\frac{z^n}{(n+1)!}=\frac{e^{z}-1}{z}.
\]

Finally, from the identity
\[
e^{z(\U+\B)}=e^{z\U}e^{z\B}=\frac{e^z-1}{z} \cdot\frac{z}{e^z-1}=1,
\]
we deduce that both symbols $\B$ and $\U$ annihilate
each other, in the sense 
\[
(x+\B+\U)^{n}=x^{n}
\]
for an arbitrary positive integer $n$, an identity that extends by
linearity to any polynomial $P$:
\begin{equation}
P(x+\B+\U)=P(x).\label{eq:Cancellation}
\end{equation}

Our main objects of study, the multiple harmonic numbers, are defined
as the truncated multiple zeta value
\[
H_{n_{1},\ldots,n_{r}}(N)=\sum_{N>i_{1}>\cdots>i_{r}>0}\frac{1}{i_{1}^{n_{1}}\cdots i_{r}^{n_{r}}},
\]
and we focus here on their values at negative indices, i.e., the multiple
power sums $H_{-n_{1},\ldots,-n_{r}}(N)$, defined in \eqref{intro}.
In particular, we shall show that these multiple power sums 
can be expressed in two equivalent ways using two new symbols which are described next. 

\subsection{The $\Hc$ symbol}

For the single harmonic sum, define, for $N\in\mathbb{N}$,
the symbol $\Hc$ by 
\[
\left(\Hc(N)\right)^{n}=H_{-n}(N)=1^{n}+2^{n}+\cdots+(N-1)^{n}.
\]
Since the Bernoulli polynomials satisfy the identity
\[
B_{n+1}(x+1)-B_{n+1}(x)=(n+1)x^{n},
\]
we have 
\begin{eqnarray}
\left(\Hc(N)\right)^{n} & = & \frac{1}{n+1}\sum_{k=0}^{N-1}(n+1)k^{n}\label{eq:Faulhaber}\\
 & = & \frac{1}{n+1}\sum_{k=0}^{N-1}\left(B_{n+1}(k+1)-B_{n+1}(k)\right)\nonumber \\
 & = & \frac{1}{n+1}\left(B_{n+1}(N)-B_{n+1}\right)\nonumber \\
 & = & \frac{\left(\B+N\right)^{n+1}-\B^{n+1}}{n+1}\nonumber \\
 & = & \frac{1}{n+1}\sum_{k=1}^{n+1}{n+1 \choose k}B_{n+1-k}N^{k},\nonumber 
\end{eqnarray}
which extends naturally to the case $\left(\Hc(x)\right)^{n}$
for $x\not\in\mathbb{N}$, as the famous Faulhaber formula.

An extension to bivariate power sums can be found in \citep[Thm.~1]{Duchamp}:
with $p=n+m+2$, 
\[
H_{-n,-m}(N)=\sum_{k=0}^{m}\ \sum_{l=0}^{p-1-k}\ \sum_{q=0}^{p-k-l}\frac{B_{k}B_{l}}{(m+1)(p-k)}{m+1 \choose k}{p-k \choose l}{p-k-l \choose q}(N-1)^{q}.
\]
This complicated triple-sum formula suggests the introduction of symbolic computation as follows: considering two independent Bernoulli symbols
$\B_{1}$ and $\B_{2}$ and summing over $q$, this
identity can be simplified as
\begin{eqnarray*}
H_{-n,-m}(N) & = & \sum_{k=0}^{m}\sum_{l=0}^{p-1-k}\frac{\B_{2}^{k}\B_{1}^{l}}{(m+1)(p-k)}{m+1 \choose k}{p-k \choose l}N^{p-k-l}\\
 & = & \frac{1}{m+1}\sum_{k=0}^{m}{m+1 \choose k}\B_{2}^{k}\frac{\left(\B_{1}+N\right)^{p-k}-\B_{1}^{p-k}}{p-k}.
\end{eqnarray*}
Therefore, using Faulhaber's formula \eqref{eq:Faulhaber}, we obtain
\begin{equation}
H_{-n,-m}(N)=\frac{1}{m+1}\sum_{k=0}^{m}{m+1 \choose k}\B_{2}^{k}\left(\Hc(N)\right)^{p-k-1}.\label{eq:Depth2Recurrence}
\end{equation}
We further deduce, by denoting 
\[
\Hc_{1}=\Hc(N),\thinspace\thinspace\thinspace\thinspace\Hc_{1,2}=\Hc\left(\Hc_{1}\right),
\]
that
\[
H_{-n,-m}(N)=\Hc_{1}^{n}(N)\cdot\frac{\left(\B_{2}+\Hc_{1}\right)^{m+1}-\B_{2}^{m+1}}{m+1}=\Hc_{1}^{n}\cdot\Hc_{1,2}^{m}.
\]
The general case is given in Theorem~\ref{thm:Hgeneralcase} below.

\subsection{The $\V$ symbol}

In order to obtain another symbolic expression,
we first consider the double sum case case, i.e., 
\[
H_{-n,-m}(N)=\sum_{N>j>i>0}i^{m}j^{n}.
\]
Introduce the ``uniform over $[0,z]$'' symbol $\V(z)$
as follows: for any polynomial $P$, 
\begin{equation}
P\left(x+\V(z)\right)=\int_{0}^{z}P(x+v)\mathrm{d}v\label{eq:Vj}
\end{equation}
so that 
\[
P\left(x+\V(z)\right)=Q(x+z)-Q(x),
\]
where $Q$ is an antiderivative of $P$. By Faulhaber's formula \eqref{eq:Faulhaber},
the inner sum (to which we have added the null term corresponding
to $i=0$) in $H_{-n,-m}(N)$ reads 
\begin{equation}
\sum_{i=0}^{j-1}i^{m}=\frac{\left(\B_{2}+j\right)^{m+1}-\left(\B_{2}+0\right)^{m+1}}{m+1}=\left(\B_{2}+\V(j)\right)^{m}.\label{eq:SingleV}
\end{equation}
Thus, 
\[
H_{-n,-m}(N)=\sum_{j=1}^{N-1}\left(\sum_{i=0}^{j-1}i^{m}\right)j^{n}=\sum_{j=1}^{N-1}j^{n}\left(\B_{2}+\V(j)\right)^{m}.
\]
Using the binomial formula yields
\[
H_{-n,-m}(N)=\sum_{j=1}^{N-1}j^{n}\sum_{k=0}^{m}{m \choose k}\B_{2}^{m-k}\left(\V(j)\right)^{k}=\sum_{k=0}^{m}{m \choose k}\B_{2}^{m-k}\left(\sum_{j=1}^{N-1}j^{n}\left(\V(j)\right)^{k}\right).
\]
Choosing $P(x)=x^{k}$, $z=j$ in \eqref{eq:Vj} and then
evaluating at $x=0$, we obtain
\[
\left(\V(j)\right)^{k}=\int_{0}^{j}v^{k}\mathrm{d}v=\frac{j^{k+1}}{k+1},
\]
so that
\[
\sum_{j=1}^{N-1}j^{n}\left(\V(j)\right)^{k}=\sum_{j=1}^{N-1}\frac{j^{n+k+1}}{k+1}=\frac{\left(\B_{1}+\V(N)\right)^{n+1+k}}{k+1}.
\]
Therefore, substitution yields 
\[
H_{-n,-m}(N)=\frac{1}{m+1}\sum_{k=0}^{m}{m+1 \choose k+1}\B_{2}^{m-k}\left(\B_{1}+\V(N)\right)^{n+1+k},
\]
where we have used the identity
\[
\frac{1}{m+1}\binom{m+1}{k+1}=\frac{1}{k+1}\binom{m}{k}.
\]
Reindexing the sum ($l=k+1$) gives
\begin{align*}
H_{-n,-m}(N) & =\frac{\left(\B_{1}+\V(N)\right)^{n}}{m+1}\sum_{l=1}^{m+1}{m+1 \choose l}\B_{2}^{m+1-l}\left(\B_{1}+\V(N)\right)^{l}\allowdisplaybreaks\\
 & =\left(\B_{1}+\V(N)\right)^{n}\frac{\left(\B_{2}+\B_{1}+\V(N)\right)^{m+1}-\B_{2}^{m+1}}{m+1}
\end{align*}
so that, defining the nested symbols 
\[
\V_{1}=\V(N)\text{ and }\V_{1,2}=\V\left(\B_{1}+\V_{1}\right),
\]
we obtain
\begin{equation}
\label{H-n-m}
H_{-n,-m}(N)=\left(\B_{1}+\V_{1}\right)^{n}\left(\B_{2}+\V_{1,2}\right)^{m}=B_n\left( \V_{1} \right)B_m\left( \V_{1,2} \right).
\end{equation}

The general case is provided in Theorem~\ref{thm:Hgeneralcase} below.

\subsection{General polynomial version}

The former results extend naturally to polynomials as follows: given
$P$ and $Q$ two polynomials without constant terms
\[
P(x)=\sum_{k=1}^{m}a_{k}x^{k}\text{ and }Q(x)=\sum_{l=1}^{n}b_{l}x^{l},
\]
a polynomial version of identity \eqref{H-n-m} reads
\begin{align*}
\sum_{N>j>i>0}P(j)Q(i) & =\sum_{k=1}^{m}\sum_{l=1}^{n}\sum_{N>j>i>0}a_{k}b_{l}i^{l}j^{k}\\
 & =\sum_{k=1}^{m}\sum_{l=1}^{n}a_{k}b_{l}\Hc_{1}^{k}\Hc_{1,2}^{l}\\
 & =\left(\sum_{k=1}^{m}a_{k}\Hc_{1}^{k}\right)\left(\sum_{l=1}^{n}b_{l}\Hc_{1,2}^{l}\right)\\
 & =P\left(\Hc_{1}\right)Q\left(\Hc_{1,2}\right),
\end{align*}
or equivalently
\[
\sum_{N>j>i>0}P(j)Q(i)=P\left(\B_{1}+\V_{1}\right)Q\left(\B_{2}+\V_{1,2}\right).
\]

\section{\label{sec:MainResults}Multiple power sums}

\subsection{Symbolic expression of multiple power sums}
\begin{thm}
\label{thm:Hgeneralcase}The $r$-fold multiple power sums \eqref{intro} can be expressed as
\begin{equation}
H_{-n_{1},\dots,-n_{r}}(N)=\prod_{k=1}^{r}\Hc_{1,\ldots,k}^{n_{k}},\label{eq:HGeneralCase}
\end{equation}
where $\Hc_{1}=\Hc(N)$ and recursively
$\Hc_{1,\ldots k}=\Hc\left(\Hc_{1,\ldots,k-1}\right)$, for $k=2,3,\ldots,r$. Equivalently, defining
\begin{equation}
\V_{1}=\V(N)\text{ and }\V_{1,\dots,k}=\V\left(\B_{k-1}+\V_{1,\dots,k-1}\right)\label{eq:NestedVSymbol}
\end{equation}
produces
\begin{equation}
H_{-n_{1},\dots,-n_{r}}(N)=\prod_{k=1}^{r}\left(\B_{k}+\V_{1,\dots,k}\right)^{n_{k}}.\label{eq:RfoldHarmonicSum}
\end{equation}
These identities extend to the case of $r$ polynomials $P_{1},\dots,P_{r}$ without
constant terms as follows
\[
\sum_{N>i_{1}>\cdots>i_{r}>0}P_{1}(i_{1})\dots P_{r}(i_{r})=\prod_{k=1}^{r}P_{k}\left(\Hc_{1,\dots,k}\right)=\prod_{k=1}^{r}P_{k}\left(\B_{k}+\V_{1,\ldots,k}\right).
\]
\end{thm}

\begin{rem*} 
{\ }
\begin{enumerate}
\item Both identities \eqref{eq:HGeneralCase} and \eqref{eq:RfoldHarmonicSum} still hold
when $N$ is a real number. 
\item Comparing \eqref{eq:HGeneralCase} and \eqref{eq:RfoldHarmonicSum}
shows that the symbols $\Hc_{1,\ldots,k}$ and $\V_{1,\ldots,k}$
are related as
\[
\Hc_{1,\ldots,k}=\B_{k}+\V_{1,\ldots,k}.
\]
 
\item The computation rules for both symbols $\Hc_{1,\ldots,k}$
and $\V_{1,\ldots,k}$ are reminiscent of the chain rule
for differentiation
\[
\frac{\mathrm{d}}{\mathrm{d}x}\left(f_{r}\circ\cdots\circ f_{1}(x)\right)=f_{r}'\left(f_{r-1}\circ\cdots\circ f_{1}\right)f'_{r-1}\left(f_{r-2}\circ\cdots\circ f_{1}\right)\cdots f'_{1}(x).
\]
\item This method also applies to the sums
\[
S_{-n_{1},\ldots,-n_{r}}(N)=\sum_{N\geq i_{1}\geq\cdots\geq i_{r}\geq1}i_{1}^{n_{1}}\cdots i_{r}^{n_{r}}
\]
and yields 
\[
S_{-n_{1},\ldots,-n_{r}}(N)=\prod_{k=1}^{r}\bar{\Hc}_{1,\ldots,k}^{n_{k}},
\]
where the symbol $\bar{\Hc}$ is now defined by $\bar{\Hc}_{1}=\Hc(N+1)$ and recursively
$\bar{\Hc}_{1,\ldots k}=\Hc\left(\bar{\Hc}_{1,\ldots,k-1}+1\right)$, for $k=2,3,\ldots,r$.
\end{enumerate}
\end{rem*}

\begin{proof}
It suffices to show \eqref{eq:RfoldHarmonicSum} by induction, since
\eqref{eq:HGeneralCase} is similar. From the definition \eqref{intro},
the recurrence 
\begin{equation}
H_{-n_{1},\dots,-n_{r}}(N)=\sum_{N>i_{1}>r-1}H_{-n_{2},\dots,-n_{r-1}}(i_{1})i_{1}^{n_{1}}\label{eq:Recurrence}
\end{equation}
holds. Note that when $N<r$, $H_{-n_{1},\dots,-n_{r}}(N)=0$,
since the indices cannot satisfy the relation 
\[
r>N>i_{1}>\cdots>i_{r}>0,
\]
so that the sum is empty. Therefore, we can further extend the summation
range of $i_{1}$ to $1\leq i_{1}\leq N-1$. Also, the identity
\eqref{eq:RfoldHarmonicSum} 
can be expressed, by shifting the indices, as
\[
H_{-n_{2},\dots,-n_{r-1}}(i_{1})=\prod_{k=2}^{r}\left(\B_{k}+\V'_{2,\dots,k}\right)^{n_{k}},
\]
where 
\begin{equation}
\V'_{2}=\V(i_{1})\text{ and recursively for }k=2,3,\ldots,r,\ \ \V'_{2,\ldots,k}=\V\left(\B_{k-1}+\V_{2,\ldots,k-1}\right).\label{eq:ShiftedVSymbol}
\end{equation}
In order to evaluate $H_{-n_{1},\dots,-n_{r}}(N)$, we expand the
nested $\V$ symbols to obtain, using \eqref{eq:SingleV},
a polynomial in the variable $\B_{2}+\V(i_{1})$
with coefficients $a_{k}$ that depend on the remaining symbols $\B_{3},\ldots,\B_{r}$, as
\begin{equation}
H_{-n_{2},\dots,-n_{r-1}}(i_{1})=\sum_{k=0}^{d}a_{k}\left(\B_{2}+\V(i_{1})\right)^{k}=\sum_{k=0}^{d}a_{k}\sum_{l=1}^{k+1}\binom{k+1}{l}\B_{2}^{k+1-l}i_{1}^{l}.\label{eq:Inductive}
\end{equation}
Thus, 
\begin{eqnarray*}
H_{-n_{1},\dots,-n_{r}}(N) & = & \sum_{i_{1}=1}^{N-1}\sum_{k=0}^{d}a_{k}\sum_{l=1}^{k+1}\binom{k+1}{l}\B_{2}^{k+1-l}i_{1}^{n_{1}+l}\\
 & = & \sum_{k=0}^{d}a_{k}\sum_{l=1}^{k+1}\binom{k+1}{l}\B_{2}^{k+1-l}\left(\sum_{i_{1}=1}^{N-1}i_{1}^{n_{1}+l}\right)\\
 & = & \sum_{k=0}^{d}a_{k}\sum_{l=1}^{k+1}\binom{k+1}{l}\B_{2}^{k+1-l}\left(\B_{1}+\V(N)\right)^{n_{1}+l}\\
 & = & \left(\B_{1}+\V(N)\right)^{n_{1}}\sum_{k=0}^{d}a_{k}\sum_{l=1}^{k+1}\binom{k+1}{l}\B_{2}^{k+1-l}\left(\B_{1}+\V(N)\right)^{l}\\
 & = & \left(\B_{1}+\V_{1}\right)^{n_{1}}\sum_{k=0}^{d}a_{k}\left(\B_{2}+\V_{1,2}\right)^{k}\\
 & = & \left(\B_{1}+\V_{1}\right)^{n_{1}}\prod_{k=2}^{r}\left(\B_{k}+\V{}_{1,2,\dots,k}\right)^{n_{k}}\allowdisplaybreaks\\
 & = & \prod_{k=1}^{r}\left(\B_{k}+\V{}_{1,2,\dots,k}\right)^{n_{k}},
\end{eqnarray*}
as desired.
\end{proof}

\subsection{Polylogarithmic Function}

Duchamp et al.~\citep{Duchamp} also considered the truncated polylogarithmic function
\[
\mathrm{Li}_{n_1,\ldots,n_r}(z;N):=\sum_{N>i_1>\cdots>i_r>0}\frac{z^{i_1}}{i_1^{n_1}\cdots i_r^{n_r}},
\]
and its values at negative indices as expressed by \eqref{poly}. 
An analog of our previous theorem is now derived for such truncated multiple polylogarithms. 
We first recall the definition of the Apostol-Bernoulli polynomials $B_{n}\left(x|\lambda\right)$ \citep{Apostol}, with generating function
\[
\frac{t}{\lambda e^{t}-1}e^{xt}=\sum_{n=0}^{\infty}B_{n}\left(x|\lambda\right)\frac{t^{n}}{n!}.
\]
These polynomials satisfy an extended version of Faulhaber's identity,  namely 
\[
\sum_{j=1}^{m-1}\lambda^{j}j^{n}=\frac{\lambda^{m}B_{n+1}\left(m|\lambda\right)-B_{n+1}\left(0|\lambda\right)}{n+1}.
\]
Notice that in the case $\lambda=1$, this is the usual Faulhaber identity, and the Apostol Bernoulli polynomials reduce to the Bernoulli polynomials.
Let us introduce the symbol $\mathcal{A}_{\lambda}$ defined by 
\[
\left(\mathcal{A}_{\lambda}+x\right)^{n}=\lambda^{x}B_{n}\left(x|\lambda\right),
\]
Choosing $x=0,$ we deduce
\[
\mathcal{A}_{\lambda}^{n}=B_{n}\left(0|\lambda\right)\Rightarrow e^{\mathcal{A}_{\lambda}t}=\sum_{n=0}^{\infty}B_{n}\left(0|\lambda\right)\frac{t^{n}}{n!}=\frac{t}{\lambda e^{t}-1};
\]
while  in the  general case
\[
e^{\mathcal{A}_{\lambda}t}e^{xt}=e^{\left(\mathcal{A}_{\lambda}+x\right)t}=\sum_{n=0}^{\infty}\lambda^{x}B_{n}\left(x|\lambda\right)\frac{t^{n}}{n!}=\lambda^{x}\frac{t}{\lambda e^{t}-1}e^{xt}.
\]
The main justification for introducing this symbol is the following identity
\begin{align*}
\left(\mathcal{A}_{\lambda}+\V(N)\right)^{n} & =\int_{0}^{N}\left(\mathcal{A}_{\lambda}+v\right)^{n}\mathrm{d}v\\
 & =\frac{\left(\mathcal{A}_{\lambda}+v\right)^{n+1}}{n+1}\bigg|_{v=0}^{v=N}\\
 & =\frac{\left(\mathcal{A}_{\lambda}+N\right)^{n+1}-\left(\mathcal{A}_{\lambda}\right)^{n+1}}{n+1}\\
 & =\frac{\lambda^{N}B_{n+1}\left(N|\lambda\right)-B_{n+1}\left(0|\lambda\right)}{n+1}\\
 & =\sum_{j=1}^{N-1}\lambda^{j}j^{n}.
\end{align*}

This identity produces the following symbolic expression for the truncated multiple polylogarithms.
\begin{thm}
Let $\bar{\V}_{1}:=\V(N)$,
$\bar{\V}_{1,2}=\V\left(\mathcal{A}_{z}+\bar{\V}_{1}\right)$,
and $\mathcal{\bar{V}}_{1,\ldots,k}=\V\left(\B_{k}+\mathcal{\bar{V}}_{1,\ldots,k-1}\right)$
for $k=3,4,\ldots,r$. Then, we have
\[
\mathrm{Li}_{-n_{1},\ldots,-n_{r}}(z;N)=\left(\mathcal{A}_{z}+\mathcal{\bar{V}}{}_{1}\right)^{n_{1}}\prod_{k=2}^{r}\left(\B_{k}+\mathcal{\bar{V}}{}_{1,2,\dots,k}\right)^{n_{k}}.
\]
Note that this expression coincides with \eqref{eq:RfoldHarmonicSum} \eqref{eq:RfoldHarmonicSum}, except for the replacement of $\B_{1}$ by $\mathcal{A}_{z}$. 
\end{thm}
\begin{proof}
By \eqref{eq:RfoldHarmonicSum},
\begin{align*}
\mathrm{Li}_{-n_{1},\ldots,-n_{r}}(z;N) & =\sum_{N>i_{1}>0}i_{1}^{n_{1}}z^{i_{1}}\left(\sum_{i_{1}>i_{2}>\cdots>i_{r}>0}i_{2}^{n_{2}}\cdots i_{r}^{n_{r}}\right)\\
 & =\sum_{N>i_{1}>0}i_{1}^{n_{1}}z^{i_{1}}H_{-n_{2},\dots,-n_{r}}(i_1)\allowdisplaybreaks\\
 & =\sum_{N>i_{1}>0}i_{1}^{n_{1}}z^{i_{1}}\prod_{k=2}^{r}\left(\B_{k}+\V'_{2,\dots,k}\right)^{n_{k}}.
\end{align*}
Using \eqref{eq:Inductive}, we have 
\begin{align*}
\mathrm{Li}_{-n_{1},\ldots,-n_{r}}(z;N) & =\sum_{N>i_{1}>0}i_{1}^{n_{1}}z^{i_{1}}\sum_{k=0}^{d}a_{k}\sum_{l=1}^{k+1}\binom{k+1}{l}\B_{2}^{k+1-l}i_{1}^{l}\\
 & =\sum_{k=0}^{d}a_{k}\sum_{l=1}^{k+1}\binom{k+1}{l}\B_{2}^{k+1-l}\left(\sum_{N>i_{1}>0}i_{1}^{n_{1}+l}z^{i_{1}}\right)\\
 & =\sum_{k=0}^{d}a_{k}\sum_{l=1}^{k+1}\binom{k+1}{l}\B_{2}^{k+1-l}\left(\mathcal{A}_{z}+\V(N)\right)^{n_{1}+l}\\
 & =\left(\mathcal{A}_{z}+\V(N)\right)^{n_{1}}\sum_{k=0}^{d}a_{k}\sum_{l=1}^{k+1}\binom{k+1}{l}\B_{2}^{k+1-l}\left(\mathcal{A}_{z}+\V(N)\right)^{l}.
\end{align*}
The proof is completed by further simplification, in the same way as in the proof of Theorem \ref{thm:Hgeneralcase}. 
\end{proof}
\begin{rem*}
For the multivariate version of these sums, a similar calculation produces 
\begin{equation}
\sum_{N>i_{1}>\cdots>i_{r}>0}i_{1}^{n_{1}}\cdots i_{r}^{n_{r}}z_{1}^{i_{1}}\cdots z_{r}^{i_{r}}=\prod_{k=1}^{r}\left(\mathcal{A}_{z_k}+\tilde{\V}_{1,\ldots,k}\right)^{n_{k}},\label{eq:MultipleLi}
\end{equation}
where $\tilde{\V}_{1}=\V(N)$ and recursively
$\tilde{\V}_{1,\ldots,k+1}=\V\left(\mathcal{A}_{z_k}+\tilde{\V}_{1,\ldots,k}\right)$,
for $k=2,3,\ldots,r$. Notice that when $z_{i}=1$, $\mathcal{A}_{z_{i}}=\B_{k}$ and identities \eqref{eq:MultipleLi} and \eqref{eq:RfoldHarmonicSum} coincide.
\end{rem*}

\subsection{Recurrence}

The symbolic representation \eqref{eq:HGeneralCase} is now exploited to  deduce the following
recurrence identity on harmonic sums.
\begin{thm}
The $r$-fold multiple power sums satisfy the recurrence
\begin{equation}
H_{-n_{1},\dots,-n_{r}}(N)=\frac{1}{n_{r}+1}\sum_{k=1}^{n_{r}+1}{n_{r}+1 \choose k}B_{n_{r}+1-k}H_{-n_{1},\dots,-n_{r-1}-k}(N).\label{eq:Recurrence1}
\end{equation}
\end{thm}

\begin{rem*}
This identity can be seen as a multivariate generalization of identity \eqref{eq:Depth2Recurrence},
and is different from identity \eqref{eq:Recurrence}.
\end{rem*}

\begin{proof}
A straightforward computation produces
\begin{eqnarray*}
H_{-n_{1},\dots,-n_{r}}(N) & = & \prod_{l=1}^{r}\Hc_{1,\ldots,l}^{n_{l}}\\
 & = & \left(\prod_{l=1}^{r-1}\Hc_{1,\ldots,l}^{n_{l}}\right)\left(\Hc\left(\Hc_{1,\ldots,r-1}\right)\right)^{n_{r}}\\
 & = & \left(\prod_{l=1}^{r-1}\Hc_{1,\ldots,l}^{n_{l}}\right)\frac{1}{n_{r}+1}\sum_{k=1}^{n_{r}+1}{n_{r}+1 \choose k}\B_{r}^{n_{r}+1-k}\Hc_{1,\ldots,r-1}^{k}\\
 & = & \frac{1}{n_{r}+1}\sum_{k=1}^{n_{r}+1}{n_{r}+1 \choose k}\B_{r}^{n_{r}+1-k}\left(\left(\prod_{l=1}^{r-2}\Hc_{1,\ldots,l}^{n_{l}}\right)\Hc_{1,\ldots,r-1}^{n_{r-1}+k}\right)\allowdisplaybreaks\\
 & = & \frac{1}{n_{r}+1}\sum_{k=1}^{n_{r}+1}{n_{r}+1 \choose k}B_{n_{r}+1-k}H_{-n_{1},\dots,-n_{r-1}-k}(N),
\end{eqnarray*}
as desired.
\end{proof}

\subsection{Generating function}
\begin{thm}
\label{thm:The-generating-function}The generating function of the $r$-fold
harmonic  sums, defined as
\[
\F_{r}\left(w_{1},\dots,w_{r};N\right)=\sum_{n_{1},\dots,n_{r}=0}^{\infty}\frac{w_{1}^{n_{1}}\dots w_{r}^{n_{r}}}{n_{1}!\dots n_{r}!}H_{-n_{1},\dots,-n_{r}}(N),
\]
satisfies, for $n\ge2,$ the recurrence
\[
\F_{r}\left(w_{1},\dots,w_{r};N\right)=\frac{\F_{r-1}\left(w_{1},\dots,w_{r-1}+w_{r};N\right)-\F_{r-1}\left(w_{1},\dots,w_{r-1};N\right)}{e^{w_{r}}-1}
\]
with the initial value
\[
\F_{1}\left(w_{1};N\right)=\frac{e^{Nw_{1}}-1}{e^{w_{1}}-1}.
\]
\end{thm}

\begin{proof}
Starting from
\[
\F_{r}\left(w_{1},\dots,w_{r};N\right)=\sum_{n_{1},\dots,n_{r}=0}^{\infty}\prod_{l=1}^{r}\left(\frac{w_l\left(\B_{l}+\V_{1,\dots,l}\right)}{n_l!}\right)^{n_{l}}=\prod_{l=1}^{r}e^{w_{l}\left(\B_{l}+\V_{1,\dots,l}\right)},
\]
we expand the last factor, using the integration rule  \eqref{eq:Vj} 
and the recurrence \eqref{eq:NestedVSymbol}, to obtain
\[
e^{w_{r}\left(\B_{r}+\V_{1,\dots,r}\right)}=\frac{e^{w_{r}\B_{r}}}{w_{r}}\left(e^{w_{r}\left(\B_{r-1}+\V_{1,\dots,r-1}\right)}-1\right),
\]
where, from \eqref{eq:BernoulliGF},
\[
\frac{e^{w_{r}\B_{r}}}{w_{r}}=\frac{1}{e^{w_{r}}-1}.
\]
Further simplification completes the proof, and computation of the initial value $\F_{1}$ is elementary.
\end{proof}

\section{\label{sec:EBP}Extended Bernoulli polynomials}

\subsection{Definitions}

Duchamp et al.~\citep{Duchamp} expressed the multiple power sums
in terms of extended Bernoulli polynomials with multiple indices $B_{n_{1},\dots,n_{r}}(z)$
as follows

\[
H_{-n_{1},\dots,-n_{r}}(N)=\frac{B_{n_{1}+1,\ldots,n_{r}+1}\left(N+1\right)-\sum_{k=1}^{r-1}b'_{n_{k}+1,\ldots,n_{r}+1}B_{n_{1}+1,\ldots,n_{k}+1}(N+1)}{\prod_{i=1}^{r}(n_{i}+1)}.
\]
where $b'_{n_1,\dots,n_r}$ is a sequence of numbers defined recursively (see \citep[Definition 1]{Duchamp}) and the extended Bernoulli polynomials $B_{n_{1},\ldots,n_{r}}\left(z\right)$ are defined recursively as follows.
\begin{defn}
For $z\in\mathbb{C}$, 
\begin{equation}
B_{n_{1},\ldots,n_{r}}\left( z+1 \right)=B_{n_{1},\ldots,n_{r}}\left( z \right)+n_{1}z^{n_{1}-1}B_{n_{2},\dots,n_{r}}\left( z \right)\label{eq:EBP}
\end{equation}
where the simple index polynomial $B_{n}\left( z \right)$ coincides with the usual Bernoulli polynomial
of degree $n$.
It appears that the recursive rule \eqref{eq:EBP}
does not allow to determine the value $B_{n_{1},\ldots,n_{r}}(0);$
however, this value is not needed anywhere in the forthcoming results.
Hence we define the shifted extended Bernoulli polynomials (without
constant term) as
\[
\beta_{n_{1},\dots,n_{r}}(z)=B_{n_{1},\ldots,n_{r}}(z)-B_{n_{1},\ldots,n_{r}}(0),
\]
and notice that they satisfy the same recurrence as the extended Bernoulli polynomials, namely
\begin{equation}
\beta_{n_{1},\ldots,n_{r}}(z+1)=\beta_{n_{1},\ldots,n_{r}}(z)+n_{1}z^{n_{1}-1}\beta_{n_{2},\dots,n_{r}}(z).\label{eq:BetaRecurrence}
\end{equation}
\end{defn}
In the next section, an explicit symbolic expression for these polynomials
is derived.

\subsection{Symbolic expression}

We use the following result that can be found as Lemma 3 in \citep{Duchamp}.
\begin{lem}
\label{lem:lemma9} Consider the difference equation 
\begin{equation}
\label{equation_Duchamp}
f\left( x+1 \right)-f\left( x \right)=P\left( x \right),
\end{equation}
where $P$ is a polynomial 
and $f:\mathbb{N}\rightarrow \mathbb{R}$ is an unknown function. Let
$P\left( x \right)=\sum_{j=0}^{d}a_{j}{x \choose j}$ be the decomposition
of $P$ in the polynomial basis $\left\{ {x \choose j}\right\} _{j\geq0}$.
Then the unique solution $f$  without constant term of equation \eqref{equation_Duchamp} is $f_{0}\left( x \right)=\sum_{j=0}^{d}a_{j}{x \choose j+1}$. 
\end{lem}

Before stating the general case, we provide the explicit symbolic
computation of the single-index and double-index shifted extended
Bernoulli polynomials. 
\begin{example}
For a single index, since $B_{n}(z)=\left(\B+z\right)^{n}$
is the usual Bernoulli polynomial, identity \eqref{eq:SingleV} shows
that
\[
\beta_{n}(z)=\left(\B+z\right)^{n}-\B^{n}=\sum_{k=1}^{n}{n \choose k}\B^{n-k}z^{k}=n\left(\B+\V(z)\right)^{n-1}.
\]
For the double-index case, from
\[
\beta_{m,n}(z+1)-\beta_{m,n}(z)=mz^{m-1}\beta_{n}(z),
\]
we express the polynomial on the right-hand side as 
\[
P(z)=mz^{m-1}n\left(\B+\V(z)\right)^{n-1}=mn\sum_{k=1}^{n}{n \choose k}\B^{n-k}z^{m-1+k}.
\]
The Stirling numbers of the second kind $\genfrac{\{}{\}}{0pt}{}{n}{k}$
satisfy 
\[
z^{m-1+k}=\sum_{l=0}^{m-1+k}\genfrac{\{}{\}}{0pt}{}{m-1+k}{l}l!{z \choose l},
\]
which implies
\begin{eqnarray*}
P(z)=\sum_{k=1}^{n}nm{n \choose k}\B^{n-k}\sum_{l=0}^{m-1+k}\genfrac{\{}{\}}{0pt}{}{m-1+k}{l}l!{z \choose l}
\end{eqnarray*}
and therefore, using Lemma~\ref{lem:lemma9}, 
\[
\beta_{m,n}(z)=\sum_{k=1}^{n}nm{n \choose k}\B^{n-k}\sum_{l=0}^{m-1+k}\genfrac{\{}{\}}{0pt}{}{m-1+k}{l}l!{z \choose l+1}.
\]
Now since
\[
\binom{z}{l+1}=\frac{(-1)^{l+1}}{(l+1)!}(-z)_{l+1},
\]
we use the identity \citep[Entry 52.2.33]{Hansen}
\[
\sum_{l=0}^{p}\genfrac{\{}{\}}{0pt}{}{p}{l}l!\frac{(-1)^{l+1}}{(l+1)!}(-z)_{l}=-\frac{B_{p+1}(z+1)-B_{p+1}}{(z+1)(p+1)}
\]
to obtain
\[
\sum_{l=0}^{p}\genfrac{\{}{\}}{0pt}{}{p}{l}l!\frac{(-1)^{l+1}}{(l+1)!}(-z-l)_{l+1}=\frac{B_{p+1}(z+1)-B_{p+1}}{p+1},
\]
and
\[
\sum_{l=0}^{p}\genfrac{\{}{\}}{0pt}{}{p}{l}l!\frac{(-1)^{l+1}}{(l+1)!}(-z)_{l+1}=\frac{B_{p+1}(z)-B_{p+1}}{p+1}=\left(\B+\V(z)\right)^{p}.
\]
Then, we deduce
\begin{eqnarray*}
\beta_{m,n}(z) & = & \sum_{k=1}^{n}mn{n \choose k}\B_{2}^{n-k}\left(\B_{1}+\V(z)\right)^{m+k-1}\allowdisplaybreaks\\
 & = & mn\left(\B_{1}+\V(z)\right)^{m-1}\sum_{k=1}^{n}{n \choose k}\B_{2}^{n-k}\left(\B_{1}+\V(z)\right)^{k}\\
 & = & mn\left(\B_{1}+\V(z)\right)^{m-1}\left(\B_{2}+\V\left(\B_{1}+\V(z)\right)\right)^{n-1}\\
 & = & mn\left(\B_{1}+\V(z)\right)^{m-1}\left(\B_{2}+\V_{1,2}\right)^{n-1}.
\end{eqnarray*}
\end{example}

The general case is given next.
\begin{thm}
For $n_{1},\dots,n_{r}>0,$ the shifted extended Bernoulli polynomials
are expressed symbolically as the product
\begin{equation}
\beta_{n_{1},\dots,n_{r}}(z)=\prod_{k=1}^{r}n_{k}\left(\B_{k}+\V_{1,\dots,k}\right)^{n_{k}-1}.\label{eq:beta}
\end{equation}
\end{thm}

Comparing with \eqref{eq:RfoldHarmonicSum}, we deduce the link
\[
\beta_{n_{1},\dots,n_{r}}(z)=\left(\prod_{k=1}^{r}\frac{\partial}{\partial\B_{k}}\right)H_{-n_{1},\dots,-n_{r}}(z)
\]
between shifted extended Bernoulli polynomials and $r$-fold multiple
power sums.
\begin{proof}
It suffices to show that 
the right hand side of \eqref{eq:beta} satisfies the recurrence \eqref{eq:BetaRecurrence}. We start with the straightforward
identity for the derivative of the $\V$ symbol 
\begin{equation}
\frac{\mathrm{d}}{\mathrm{d}z}\left(\left(\B+\V(z)\right)^{n}\right)=\frac{\mathrm{d}}{\mathrm{d}z}\left(\frac{\left(\B+z\right)^{n+1}-\B^{n+1}}{n+1}\right)=\left(\B+z\right)^{n},\label{eq:derivative}
\end{equation}
which can be described as the ``differentiation replaces $\V(z)$
by $z$'' rule.

In order to compute
\[
\beta_{n_{1},\dots,n_{r}}(z)=n_{1}\cdots n_{r}\left(\B_{1}+\V_{1}\right)^{n_{1}-1}\cdots\left(\B_{r}+\V_{1,\ldots,r}\right)^{n_{r}-1},
\]
similarly to the case of multiple power sums, we expand
\begin{equation}
\beta_{n_{1},\dots,n_{r}}(z)=\sum_{k=0}^{d}a_{k}\left(\B_{1}+\V(z)\right)^{k},\label{eq:expansionBeta}
\end{equation}
a polynomial in $\B_{1}+\V(z)$
with coefficients $a_{k}$ that depend on the symbols $\B_{2},\ldots,\B_{r}$.
Thus, using \eqref{eq:derivative}, we obtain the derivative
\[
\beta'_{n_{1},\dots,n_{r}}(z)=\sum_{k=0}^{d}a_{k}\left(\B_{1}+z\right)^{k},
\]
from which, replacing $z$ by $z+\U$ and applying the cancellation
property \eqref{eq:Cancellation}, we deduce, by using \eqref{eq:expansionBeta}
and \eqref{eq:ShiftedVSymbol},
\[
\beta'_{n_{1},\dots,n_{r}}\left(z+\U\right)=\sum_{k=0}^{d}a_{k}z^{k}=n_{1}\cdots n_{r}z^{n_{1}-1}\left(\B_{2}+\V'_{2}(z)\right)^{n_{2}-1}\cdots\left(\B_{r}+\V'_{2,\ldots,r}\right)^{n_{r}-1},
\]
which produces the desired recurrence
\[
\beta_{n_{1},\dots,n_{r}}(z+1)-\beta_{n_{1},\dots,n_{r}}(z)=\beta'_{n_{1},\dots,n_{r}}\left(z+\U\right)=n_{1}z^{n_{1}-1}\beta{}_{n_{2},\dots,n_{r}}(z).\qedhere
\]
\end{proof}
\begin{rem*}
Recurrence \eqref{eq:EBP} does not determine the constant term for
all extended Bernoulli polynomials except for the single-indexed ones,
which coincide with the usual Bernoulli polynomials. Therefore, if
we alternatively define 
\[
\tilde{\beta}_{n_{1},\ldots,n_{r}}(z)=\begin{cases}
B_{n_{1}}(z), & \text{if }r=1;\\
\beta_{n_{1},\ldots,n_{r}}(z), & \text{if }r>1,
\end{cases}
\]
a similar proof produces
\[
\tilde{\beta}_{n_{1},\ldots,n_{r}}(z)=\left(\prod_{k=1}^{r-1}n_{k}\left(\B_{k}+\V_{1,\dots,k}\right)^{n_{k}-1}\right)\left(\B_{r}+\B_{r-1}+\V_{1,\ldots,r-1}\right)^{n_{r}}.
\]
\end{rem*}

\subsection{Generating function}
\begin{example}
Let us first compute the generating function of the extended Bernoulli polynomials with two indices
$\beta_{n_{1},n_{2}}(z)$ as follows: with 
\end{example}

\[
\beta_{n_{1},n_{2}}(z)=n_{1}n_{2}\left(\B_{1}+\V_{1}\right)^{n_{1}-1}\left(\B_{2}+\V_{1,2}\right)^{n_{2}-1},
\]
define the generating function
\[
\G_{2}\left(w_{1},w_{2};z\right):=\sum_{n_{1},n_{2}=0}^{\infty}\beta_{n_{1},n_{2}}(z)\frac{w_{1}^{n_{1}}w_{2}^{n_{2}}}{n_{1}!n_{2}!}.
\]
Then 
\begin{eqnarray*}
\G_{2}\left(w_{1},w_{2};z\right) & = & \sum_{n_{1},n_{2}=0}^{\infty}\left(n_{1}\left(\B_{1}+\V_{1}\right)^{n_{1}-1}n_{2}\left(\B_{2}+\V_{1,2}\right)^{n_{2}-1}\right)\frac{w_{1}^{n_{1}}w_{2}^{n_{2}}}{n_{1}!n_{2}!}\\
 & = & w_{1}e^{w_{1}\left(\B_{1}+\V_{1}\right)}\cdot w_{2}e^{w_{2}\left(\B_{2}+\V_{1,2}\right)}\\
 & = & w_{1}e^{w_{1}\left(\B_{1}+\V_{1}\right)}\cdot\left(e^{w_{2}\left(\B_{2}+\B_{1}+\V_{1}\right)}-e^{w_{2}\B_{2}}\right)\\
 & = & w_{1}e^{w\B_{2}}\left(e^{\left(w_{1}+w_{2}\right)\left(\B_{1}+\V_{1}\right)}-e^{w_{1}\left(\B_{1}+\V_{1}\right)}\right)\\
 & = & e^{w_{2}\B_{2}}\left(\frac{w_{1}}{w_{1}+w_{2}}\G_{1}\left(w_{1}+w_{2};z\right)-\G_{1}\left(w_{1};z\right)\right).
\end{eqnarray*}
where $\G_{1}\left(w_{1};z\right)$ is the generating function
of the usual shifted Bernoulli polynomials $\beta_{n}(z)=B_{n}(z)-B_{n}(0),$
which is equal to
\begin{equation}
\G_{1}\left(w_{1};z\right)=\frac{w\left(e^{zw}-1\right)}{e^{w}-1}.\label{eq:F1}
\end{equation}
The general case is as follows.
\begin{thm}
The generating function of the shifted extended Bernoulli polynomials
\[
\G_{r}\left(w_{1},\dots,w_{r};z\right)=\sum_{n_{1},\dots,n_{r}=0}^{\infty}\frac{w_{1}^{n_{1}}\dots w_{r}^{n_{2}}}{n_{1}!\dots n_{r}!}\beta_{n_{1},\dots,n_{r}}(z)
\]
satisfies the recurrence
\[
\G_{r}\left(w_{1},\dots,w_{r};z\right)=\frac{w_{r}\left(\frac{w_{r-1}}{w_{r-1}+w_{r}}\G_{r-1}\left(w_{1},\dots,w_{r-2},w_{r-1}+w_{r};z\right)-\G_{r-1}\left(w_{1},\dots,w_{r-2},w_{r-1};z\right)\right)}{e^{w_{r}}-1}
\]
with initial condition given by \eqref{eq:F1}.
\end{thm}

\begin{proof}
Using \eqref{eq:beta}, we have
\[
\G_{r}\left(w_{1},\dots,w_{r};z\right)=\sum_{n_{1},\dots,n_{r}=0}^{\infty}\frac{w_{1}^{n_{1}}\dots w_{r}^{n_{2}}}{n_{1}!\dots n_{r}!}\prod_{k=1}^{r}n_{k}\left(\B_{k}+\V_{1,\dots,k}\right)^{n_{k}-1}=\prod_{k=1}^{r}w_{k}e^{w_{k}\left(\B_{k}+\V_{1,\dots,k}\right)},
\]
where the last factor is expanded as
\[
w_{r}e^{w_{r}\left(\B_{r}+\V_{1,\dots,r}\right)}=e^{w_{r}\B_{r}}\left(e^{w_{r}\left(\B_{r-1}+\V_{1,\dots,r-1}\right)}-1\right)
\]
so that
\begin{eqnarray*}
\G_{r}\left(w_{1},\dots,w_{r};z\right) & = & e^{w_{r}\B_{r}}\left(\prod_{k=1}^{r-2}w_{k}e^{w_{k}\left(\B_{k}+\V_{1,\dots,k}\right)}\right)\\
 &  & \times w_{r-1}\left(e^{\left(w_{r-1}+w_{r}\right)\left(\B_{r-1}+\V_{1,\dots,r-1}\right)}-e^{w_{r-1}\left(\B_{r-1}+\V_{1,\dots,r-1}\right)}\right),
\end{eqnarray*}
which is precisely
\[
\frac{w_{r}}{e^{w_{r}}-1}\left(\frac{w_{r-1}}{w_{r-1}+w_{r}}\G_{r-1}\left(w_{1},\dots,w_{r-1}+w_{r};z\right)-\G_{r-1}\left(w_{1},\dots,w_{r-1};z\right)\right).
\]
This completes the proof.
\end{proof}

\subsection{Connection among the symbols $\beta$, $\B$ and $\Hc$}

The next result provides a connection formula between the symbols $\beta$, $\B$ and $\Hc$.
\begin{thm}
The symbols $\beta$, $\B$ and $\Hc$ are related by
\begin{equation}
\beta_{n_{1},\ldots,n_{r}}(N)=\beta_{n_{1}}(N)B_{n_{2},\ldots,n_{r}}(N)-n_{2}\left(\Hc(N)\right)^{n_{2}-1}B_{n_{3},\ldots,n_{r}}\left(\Hc(N)\right)\beta_{n_{1}}\left(\Hc(N)+1\right).\label{eq:Beta_h}
\end{equation}
\end{thm}

\begin{proof}
From the recurrence \eqref{eq:EBP}
\[
B_{n_{1},\ldots,n_{r}}(k+1)=B_{n_{1},\ldots,n_{r}}(k)+n_{1}k^{n_{1}-1}B_{n_{2},\dots,n_{r}}(k),
\]
we deduce, summing over $k$ from $0$ to $N-1,$
\begin{eqnarray*}
\beta_{n_{1},\ldots,n_{r}}(N) & = & B_{n_{1},\ldots,n_{r}}(N)-B_{n_{1},\ldots,n_{r}}(0)\\
 & = & \sum_{k=0}^{N-1}B_{n_{1},\ldots,n_{r}}(k+1)-B_{n_{1},\ldots,n_{r}}(k)\\
 & = & \sum_{k=0}^{N-1}n_{1}k^{n_{1}-1}B_{n_{2},\dots,n_{r}}(k).
\end{eqnarray*}
Using the summation by parts formula
\[
\sum_{k=0}^{N-1}f_{k}g_{k}=f_{N-1}\sum_{k=0}^{N-1}g_{k}-\sum_{j=0}^{N-2}\left(f_{j+1}-f_{j}\right)\sum_{k=0}^{j}g_{k}.
\]
with 
\[
f_{k}=B_{n_{2},\dots,n_{r}}(k)\text{ and }g_{k}=n_{1}k^{n_{1}-1}
\]
we deduce
\begin{eqnarray*}
\beta_{n_{1},\ldots,n_{r}}(N) & = & B_{n_{2},\dots,n_{r}}(N-1)\sum_{k=0}^{N-1}n_{1}k^{n_{1}-1}-\sum_{j=0}^{N-2}\left(B_{n_{2},\dots,n_{r}}(j+1)-B_{n_{2},\dots,n_{r}}(j)\right)\sum_{k=0}^{j}n_{1}k^{n_{1}-1}\\
 & = & B_{n_{2},\dots,n_{r}}(N)\sum_{k=0}^{N-1}n_{1}k^{n_{1}-1}-\sum_{j=0}^{N-1}\left(B_{n_{2},\dots,n_{r}}(j+1)-B_{n_{2},\dots,n_{r}}(j)\right)\sum_{k=0}^{j}n_{1}k^{n_{1}-1}\\
 & = & B_{n_{2},\dots,n_{r}}(N)\beta_{n_{1}}(j)-\sum_{j=0}^{N-1}n_{2}j^{n_{2}-1}B_{n_{3},\ldots,n_{r}}(j)\beta_{n_{1}}(j+1).
\end{eqnarray*}
Assume 
\[
B_{n_{3},\ldots,n_{r}}(j)\beta_{n_{1}}(j+1)=\sum_{l=0}^{d}a_{l}j^{l},
\]
so that 
\begin{eqnarray*}
\sum_{j=0}^{N-1}n_{2}j^{n_{2}-1}B_{n_{3},\ldots,n_{r}}(j)\beta_{n_{1}}(j+1) & = & \sum_{l=0}^{d}a_{l}\left(\Hc(N)\right)^{n_{2}+l-1}\\
 & = & n_{2}\left(\Hc(N)\right)^{n_{2}-1}B_{n_{3},\ldots,n_{r}}\left(\Hc(N)\right)\beta_{n_{1}}\left(\Hc(N)+1\right),
\end{eqnarray*}
which completes the proof.
\end{proof}
\begin{rem*}
Since $\beta_{n_{1},\ldots,n_{r}}(z)$ is a polynomial
in $z$, identity \eqref{eq:Beta_h} extends to the case of an arbitrary real number $z\in\mathbb{R}$ as
\[
\beta_{n_{1},\ldots,n_{r}}(z)=\beta_{n_{1}}(z)B_{n_{2},\ldots,n_{r}}(z)-n_{2}\left(\Hc(z)\right)^{n_{2}-1}B_{n_{3},\ldots,n_{r}}\left(\Hc(z)\right)\beta_{n_{1}}\left(\Hc(z)+1\right).
\]
\end{rem*}

\begin{example}
The special case of double-index shifted extended Bernoulli polynomials
with $n=2$ reads
\[
\beta_{2,m}(N)=\beta_{2}(N)B_{m}(N)-m\left(\Hc(N)\right)^{m-1}\beta_{n}\left(\Hc(N)+1\right).
\]
Since
\[
\beta_{2}(x)=B_{2}(x)-B_{2}=x^{2}-x,
\]
we deduce
\begin{eqnarray*}
\beta_{2,m}(N) & = & (N^2-N)B_{m}(N)-m\left(\Hc(N)\right)^{m-1}\left(\left(\Hc(N)+1\right)^{2}-\left(\Hc(N)+1\right)\right)\\
 & = & (N^2-N)B_{m}(N)-m\left(\left(\Hc(N)\right)^{m+1}+\left(\Hc(N)\right)^{m}\right)\\
 & = & (N^2-N)B_{m}(N)-m\left(H_{-m-1}(N)+H_{-m}(N)\right),
\end{eqnarray*}
so that
\[
\beta_{2,m}(N)=2\sum_{k=0}^{N-1}kB_{m}(k)=(N^2-N)B_{m}(N)-m\left(H_{-m-1}(N)+H_{-m}(N)\right).
\]
\end{example}

\section{Analytic Continuation}
Having obtained explicit or recurrence expressions for the finite case $N \in \mathbb{N}$, we shall study the limiting behavior of the harmonic  sum
$$\sum_{N>i_{1}>\cdots>i_{r}>0}i_{1}^{n_{1}}\cdots i_{r}^{n_{r}}$$
as $N \to \infty$. The series is obviously divergent in this limit, but deducing its value is equivalent to evaluating the analytic continuation of the multiple zeta function at negative integers. We introduce two natural methods to assign this divergent series a value: a natural renormalization procedure, and extension to a constant term, and show that they  coincide. In fact, we will then show that they both yield \textit{Raabe's formula} \citep{Raabe, Sadaoui}. For what follows, we will extensively use the shorthand
$$ 
\tilde{n} = \sum_{j=1}^r n_j, \tilde{k} = \sum_{j=2}^r k_j.
$$
Then, given the multiple zeta function 
$$
\zeta({s_1,s_2,\ldots,s_r}) = \sum_{n_1>n_2>\cdots>n_r} \frac{1}{n_1^{s_1}n_2^{s_2}\cdots n_r^{s_r}}.
 $$
its analytic continuation is evaluated in \citep[eq.~13]{Raabe} (with typos in the indices of the $B_{k_j}$ corrected, and the prefactor of $1/(n_r+1)$ correctly omitted), as the $(r-1)$ fold sum
\begin{align}\label{Raabe:explicit}
 \zeta(-n_1,-n_2,\ldots, -n_r) = {(-1)^{\tilde{n}}}\sum_{k_2,k_3,\ldots,k_r} \frac{B_{\tilde{n}+r-\tilde{k}}}{\tilde{n}+r-\tilde{k}} &\prod_{j=2}^r   \binom{ \sum_{i=j}^r n_i + r - j + 1 - \sum_{i=j+1}^r k_i }{k_j} \\
&\times\frac{B_{k_j}}{\sum_{i=j}^r n_i + r - j + 1 - \sum_{i=j+1}^r k_i }. \nonumber
\end{align}

\subsection{Renormalization}
For what follows, notice that, for $k\geq 2,$ the $\Hc$ symbols satisfy the evaluation rule
\begin{align*}
\Hc_{1,2,\ldots,k}^{n_k} &= (\B_k + \V_{1,\ldots,k})^{n_k} \\
&=  (\B_k + \V(\B_{k-1}+V_{1,\ldots,k-1}))^{n_k} \\
&= \int_0^{\B_{k-1}+V_{1,\ldots,k-1}} (\B_k+t)^{n_k} dt \\
&= \frac{ (\B_k+\B_{k-1} +\V_{1,\ldots,k-1})^{n_k+1} - \B_k^{n_k+1} }{ n_k+1 } \\
&= \frac{ (\B_k+\Hc_{1,\ldots,k-1})^{n_k+1} - B_{n_k+1} }{ n_k+1 }.
\end{align*}
For $k=1$, they are evaluated as
$$\Hc_1^{n} = \frac{(\B_1+N)^{n+1} - B_{n+1}}{n+1}.$$
\begin{thm}
Define the following renormalization rules for the symbol $\Hc_{1,\ldots,k}(\infty)$:
\begin{itemize}
\item for $k \geq 2$, define recursively $$\Hc_{1,2,\ldots,k}^{n_k}(\infty) =  \frac{ (\B_k+\Hc_{1,\ldots,k-1}(\infty))^{n_k+1} }{ n_k+1 };$$
\item for $k=1$, define $$\Hc_1^n(\infty) =\frac{B_{n+1}}{n+1}. $$ 
\end{itemize}
Then
\[
H_{-n_1,\ldots,-n_r}(\infty):=\prod_{k=1}^r \Hc_{1,\ldots,k}^{n_k}(\infty) = (-1)^{\tilde{n}} \zeta(-n_1,\ldots,-n_r),
\]
which, up to sign, is the value of $\zeta(-n_1,\ldots,-n_r)$ given by Raabe's analytic continuation in \citep{Raabe}.
\end{thm}
\begin{proof}
Given independent Bernoulli symbols $\{\B_1,\B_2,\ldots,\B_k\}$, define the $\C$ symbols through the recursive rule
$$\C_1^n = \frac{\B_1^n}{n}, \C_{1,2}^n = \frac{(\C_1+\B_2)^n}{n},  \cdots, \C_{1,2,\ldots,k}^n = \frac{(\C_{1,2,\ldots,k-1}+\B_k)^n}{n}.$$
The main result of \citep[Thm.~1]{Raabe} was the symbolic expression
$$\zeta(-n_1,\ldots,-n_r)  = (-1)^{\tilde{n}} \prod_{k=1}^r \C_{1,2,\ldots,k}^{n_k+1}.$$
On the other hand,  the renormalization procedure above indicates that the action of the $\Hc$ and $\C$ symbols exactly coincide; in fact, we have the equality $$\Hc_{1,\ldots,k}^{n_k}(\infty) = \C_{1,\ldots,k}^{n_k+1},$$ 
which completes the proof.
\end{proof}
Heuristically, at depth $k=1$ we discard the divergent contribution from the $N$ term and keep the constant term, while at depth $k\geq 2$ we discard the constant term. This is yet another appearance of Raabe's analytic continuation, from a natural renormalization procedure.

\subsection{Constant Term Interpretation}
We begin by studying the depth one case. There, Faulhaber's formula reads
$$  H_n(N) = (\B_1+\V_1)^n = \sum_{k=0}^{n} \binom{n+1}{k}\frac{B_k}{n+1}N^{n+1-k}. $$
Note that there is \textit{no constant term}, as expected, since $H_n^-(0)=0$ trivially. However, if such a constant term were to exist, we could evaluate it by putting $k=n+1$ in the summand, giving 
$$\binom{n+1}{n+1} \frac{B_{n+1}}{n+1}  = \frac{B_{n+1}}{n+1}  = \zeta(-n)(-1)^n.$$

We now explicitly describe the action of the $\Hc$ symbols, noting that this is a case study in how useful symbols are, since they enable the easy manipulation of multiple sums as those that follow.
\begin{thm}
We have the explicit expansion, as a polynomial in $N$,
\begin{align}\label{Faulhaber:explicit}
H_{-n_1,\ldots,-n_r}(N) &= \sum_{k_r=0}^{n_r}\thinspace\thinspace\sum_{k_{r-1}=0}^{n_r + n_{r-1} +1-k_r }\thinspace\thinspace\sum_{k_{r-2}=0}^{n_r + n_{r-1} +n_{r-2}+2-k_r-k_{r-1} } \cdots \sum_{k_1=0}^{  \sum_{i=1}^r n_i + (r-1) -\sum_{i=2}^rk_i} \\
&\medspace\medspace\medspace\medspace\medspace\medspace\prod_{j=1}^r \binom{   \sum_{i=j}^r n_i + (r-j+1) -\sum_{i=j+1}^rk_i}{ k_j} \frac{B_{k_j}}{ \sum_{i=j}^r n_i + (r-j+1) -\sum_{i=j+1}^rk_i} \nonumber\\
&\medspace\medspace\medspace\medspace\medspace\medspace\times N^{ \sum_{i=1}^r n_i + r -\sum_{i=1}^rk_i} \nonumber.
\end{align}
\end{thm}
\begin{proof}
We prove by induction on the depth $r$. For $r=1$, this is precisely Faulhaber's formula. Now assume that the formula holds for $r$, and use the inductive hypothesis to write
\begin{align*}
&H_{-n_{1},\dots,-n_{r},-n_{r+1}}(N)=(\B_1+\V_1)^{n_1}\prod_{k=2}^{r+1}\Hc_{1,\ldots,k}^{n_{k}} \\
=&(\B_1+\V_1)^{n_1}\times  \sum_{k_{r+1}=0}^{n_{r+1}}\thinspace\thinspace\sum_{k_{r}=0}^{n_{r+1} + n_{r} +1-k_{r+1} }\thinspace\thinspace\sum_{k_{r-1}=0}^{n_{r+1} + n_{r} +n_{r-1}+2-k_{r+1}-k_{r} } \cdots \sum_{k_2=0}^{  \sum_{i=2}^{r+1} n_i + (r-1) -\sum_{i=3}^{r+1} k_i} \allowdisplaybreaks \\
&\times\prod_{j=2}^{r+1} \binom{   \sum_{i=j}^{r+1} n_i + (r-j+2) -\sum_{i=j+1}^{r+1} k_i}{ k_j} \frac{B_{k_j}}{ \sum_{i=j}^{r+1} n_i + (r-j+2) -\sum_{i=j+1}^{r+1} k_i} \\
&\times (\B_1+\V_1)^{ \sum_{i=2}^{r+1} n_i + r -\sum_{i=2}^{r+1}k_i}.
\end{align*}
We then multiply together both occurrences of $(\B_1+\V_1)$ and expand 
$$(\B_1+\V_1)^{n} = \sum_{k_1=0}^n \binom{n+1}{k_1} \frac{1}{n+1}B_lN^{n+1-l}$$
with $n =  \sum_{i=1}^{r+1} n_i + r -\sum_{i=2}^{r+1}k_i $. This provides the innermost $k_1$ summation and the $j=1$ term in the product, completing the inductive proof.
\end{proof}

Note that there is no constant term in \eqref{Faulhaber:explicit}, since $H_{n_1,\ldots,n_r}(0)=0$.
\begin{thm}
If we instead sum over all tuples $(k_r,k_{r-1},\ldots, k_1)$ which would give a constant coefficient  in \eqref{Faulhaber:explicit}, we obtain Raabe's analytic continuation $(-1)^{\tilde{n}}\zeta(-n_1,\ldots,-n_r)$.
\end{thm}
\begin{proof}
 This constant term condition over the summation set translates into the equivalent condition $ \sum_{i=1}^r n_i + r =\sum_{i=1}^rk_i$, giving the sum
\begin{align*}
\sum_{\substack{ k_1,k_2,\cdots,k_r \\  \sum_{i=1}^r n_i + r =\sum_{i=1}^rk_i}}\prod_{j=1}^r \binom{   \sum_{i=j}^r n_i + (r-j+1) -\sum_{i=j+1}^rk_i}{ k_j} \frac{B_{k_j}}{ \sum_{i=j}^r n_i + (r-j+1) -\sum_{i=j+1}^rk_i}.
\end{align*}
Since $ \sum_{i=1}^r n_i + r =\sum_{i=1}^rk_i$, the binomial coefficient at $j=1$ reduces to $\binom{k_1}{k_1} = 1$. We can also omit the summation over $k_1$ by writing $k_1 =  \sum_{i=1}^r n_i + r -\sum_{i=2}^rk_i$, yielding
\[
\sum_{{k_2,\cdots,k_r}}\prod_{j=2}^r \binom{ \displaystyle\sum_{i=j}^r n_i + (r-j+1) -\sum_{i=j+1}^rk_i}{ k_j} \frac{B_{k_j}}{\displaystyle \sum_{i=j}^r n_i + (r-j+1) -\sum_{i=j+1}^rk_i}\cdot\frac{B_{ \sum_{i=1}^r n_i + r -\sum_{i=2}^rk_i  }}{ \displaystyle\sum_{i=1}^r n_i + r -\sum_{i=2}^rk_i}.
\]
By comparing with the explicit expansion \eqref{Raabe:explicit}, we see this is exactly Raabe's continuation, which completes the proof.
\end{proof}

\begin{example}
For depth two, the extended Faulhaber formula reads 
$$H_{n,m}(N) = \sum_{k=0}^m\sum_{j=0}^{n+m+1-k} \binom{m+1}{k} \frac{B_k}{m+1} \binom{n+m+2-k}{j} \frac{B_j}{n+m+2-k}N^{n+m+2-k-j}. $$
Again we want to extend this to a constant term, so set $n+m+2-k-j=0$ on the summand, giving
$$ \binom{m+1}{k} \frac{B_k}{m+1} \binom{n+m+2-k}{j} \frac{B_j}{n+m+2-k} =  \binom{m+1}{k} \frac{B_k}{m+1} \frac{B_j}{j}.$$
Hence the ``constant term'' is 
$$\sum_{k+j=2+m+n}\binom{m+1}{k} \frac{B_k}{m+1} \frac{B_j}{j} $$
Now Raabe's analytic continuation for depth $2$ reads $$(-1)^{m+n}\zeta({-n,-m}) =  \sum_{k=0}^{m+1}  \binom{m+1}{k} \frac{B_k}{m+1} \frac{B_{n+m+2-k}}{n+m+2-k},$$
as desired.
\end{example}

\section{Acknowledgment}

The first author was partially supported by the Austrian Science Fund (FWF) grant
SFB F50 (F5006-N15 and F5009-N15)

\section*{References}

\bibliography{faulhaber}

%
%
%
%
%
%
%
%
%
%
%

\end{document}